\documentclass[a4paper,11pt]{amsart}
\usepackage[english]{babel} 
\usepackage{mathtools} 
\usepackage{amssymb}
\usepackage[T1]{fontenc}
\usepackage{lmodern}
\usepackage{mathrsfs} 
\usepackage{slashed} 
\usepackage[babel=true]{microtype} 
\usepackage[autostyle=true]{csquotes} 
\usepackage[dvipsnames]{xcolor} 
\calclayout

\usepackage[biblatex=true]{embrac} 

\usepackage{tikz}
\usetikzlibrary{cd} 

\usepackage[shortlabels]{enumitem}
\setlist[enumerate,1]{label=\upshape(\arabic*)}
\newlist{myenumi}{enumerate}{1}
\setlist[myenumi,1]{label=\upshape(\roman*)}
\newlist{myenuma}{enumerate}{1}
\setlist[myenuma,1]{label=\upshape(\alph*)}

\numberwithin{equation}{section}
\allowdisplaybreaks[1]

\newtheorem{theorem}{Theorem}[section]
\newtheorem*{theorem*}{Theorem}
\usepackage{thmtools}

\definecolor{Heather}{RGB}{164, 132, 172}
\usepackage[pdfusetitle,bookmarks,pdfpagelabels]{hyperref} %
\hypersetup{
  colorlinks,
  urlcolor=Periwinkle,
  citecolor=Heather,
  linkcolor=teal,
  breaklinks=true,
  final,
}

\declaretheorem[name=Lemma, numbered=no]{lemma*}

\declaretheorem[name=Corollary, numbered=no]{corollary*}
\declaretheorem[name=Proposition, numberlike=theorem]{proposition}

\declaretheorem[name=Theorem]{theoremx}

\usepackage[noabbrev,capitalise]{cleveref} 
\crefdefaultlabelformat{#2\textup{#1}#3}
\crefname{theoremx}{Theorem}{Theorems}
\Crefname{theoremx}{Theorem}{Theorems}
\crefname{corollaryx}{Corollary}{Corollaries}
\Crefname{corollaryx}{Corollary}{Corollaries}

\usepackage{todonotes}

\makeatletter
\providecommand\@dotsep{5}
\def\listtodoname{List of Todos}
\def\listoftodos{\@starttoc{tdo}\listtodoname}
\makeatother
\NewDocumentEnvironment{mytodoenv}{m}{\hypersetup{hidelinks}\textsf{\textbf{#1:}} }{}

\usepackage{mymacros}

\title[Rigidity for spin fill-ins with scalar curvature bounded from below]{Rigidity for spin fill-ins with scalar curvature bounded from below}

\author{Bernd Ammann}

\address[Ammann]{Universität Regensburg, Fakultät für Mathematik, 93040 Regensburg, Germany}
\email{\href{mailto:bernd.ammann@mathematik.uni-regensburg.de}{bernd.ammann@mathematik.uni-regensburg.de}}

\author{Samuel Lockman}
\thanks{Both authors were funded by the Deutsche Forschungsgemeinschaft (DFG, German Research Foundation) – Project numbers 224262486; 
313840899}  

\address[Lockman]{Universität Regensburg, Fakultät für Mathematik, 93040 Regensburg, Germany}
\email{\href{mailto:samuel.lockman@mathematik.uni-regensburg.de}{samuel.lockman@mathematik.uni-regensburg.de}}

\hypersetup{
  pdfauthor={Bernd Ammann, Samuel Lockman} 
}
\begin{document}

\begin{abstract}
We establish the rigidity statement in the equality case of the hyperspherical-radius inequality of Brendle, Tsiamis, and Wang for compact spin fill-ins with scalar curvature bounded below \cite{brendle-fill-in}. More precisely, let $(M^{n\geq 3},g)$ be a compact, connected Riemannian spin manifold having a connected boundary $\Sigma$ and scalar curvature satisfying $\scal_g\geq -n(n-1)$. We prove that equality in the upper bound
\[
    \inf_{\Sigma}H\leq (n-1)\sqrt{1+\operatorname{Rad}(\Sigma)^{-2}}
\]
given by Brendle, Tsiamis, and Wang holds if and only if $(M,g)$ is isometric to a geodesic ball in hyperbolic space.
\end{abstract}
    
\maketitle
\section{Introduction}

A basic question in scalar curvature geometry is to ask to what extent a lower bound for the scalar curvature of a manifold with boundary controls the mean curvature of its boundary. More specifically, consider the following question. Given a closed Riemannian manifold $(\Sigma,\gamma)$, a smooth function $H\colon\Sigma \to \R$, and a constant $C\in\R$, one may ask whether there exists a compact Riemannian manifold $(M,g)$ such that
\[
    (\bd M,g|_{\bd M})\cong(\Sigma, \gamma),\qquad
    \mean_{\bd M}=H,\qquad
    \scal_g\geq C.
\]
Such a manifold is called a fill-in of the boundary data $(\Sigma,\gamma,H)$ with scalar curvature bounded below by $C$. For work related to this kind of problem, see \cite{Shi-Tam, Eichmair, JJ,Mantoulidis, Miao, Shi-Wang-Wei-2021, Shi-Wang-Wei-2022, McCormick, cecchinihirschzeidler, baer2026upperboundtotalmean, frenck2026surgerytotalmeancurvature, wang2024fillinstoriscalarcurvature}.

Directly related to this article is a result due to Hijazi, Montiel, and Roldán \cite{Hijazi}, who showed the following. 
\begin{theorem}[Hijazi, Montiel, Roldán]\label{Hijazi-bound}
    Let $(M^n,g)$ be a compact, connected Riemannian spin manifold of
    dimension $n \geq 3$ with connected boundary $\Sigma$. If $\scal_g \geq -n(n-1)$, then
    \[
        \inf_{\bd M} \mean
        \leq
        \sqrt{(n-1)^2+4\lambda^2},
    \]
    where $\lambda = \sqrt{\lambda_{\min}\bigl(\ReducedSpinDirac^2\bigr)}$
    and $\ReducedSpinDirac$ denotes the Dirac operator on the boundary.
\end{theorem}
The hyperspherical radius, introduced by Gromov and Lawson in their work on Geroch's conjecture \cite{Gromov-Lawson-80}, is defined by
\[
    \rad(\Sigma)
    :=
    \sup\left\{
        r>0 \ \middle|\
        \begin{array}{c}
            \text{there exists a $1$-Lipschitz map }
            f\colon \Sigma\longrightarrow \Sphere^{n-1}(r)\\
            \text{of nonzero degree}
        \end{array}
    \right\}.
\]
This invariant is closely related to scalar curvature. Indeed, Llarull's theorem \cite{Llarull} may be viewed as giving an upper bound for the infimum of the scalar curvature of a closed spin manifold in terms of its hyperspherical radius. Recently, Bär \cite{Bar-Dirac-Eigen} proved the corresponding spectral estimate
\begin{equation}\label{bar-bound}
    \lambda
    =
    \sqrt{\lambda_{\min}\bigl(\ReducedSpinDirac^2\bigr)}
    \leq
    \frac{n-1}{2}\rad(\Sigma)^{-1},   
\end{equation}
which allowed Bär to give short and elegant proofs of both Geroch's conjecture and Llarull's theorem for spin manifolds.
Brendle, Tsiamis, and Wang~\cite{brendle-fill-in} observed that combining this estimate with Theorem \ref{Hijazi-bound} proves an inequality conjectured by Gromov \cite[page 110]{Four-lectures-Gromov}. Part of the contribution by Brendle, Tsiamis, and Wang \cite{brendle-fill-in} was also to give a new proof of Theorem \ref{Hijazi-bound}. Specifically, they prove the following.
\begin{theorem}[Brendle, Tsiamis, Wang]
    Let $(M^n, g)$ be a compact, connected Riemannian spin manifold of dimension $n \geq 3$ with connected boundary $\Sigma$. If $\scal_g \geq -n(n-1)$, then 
    \begin{equation}\label{mean-curv-ineq}
        \inf_{\Sigma} \mean \leq (n-1) \sqrt{1 + \mathrm{Rad}(\Sigma)^{-2}}.
    \end{equation}
\end{theorem}
The goal of this article is to address the equality case of inequality~\eqref{mean-curv-ineq}. We will prove the following.
\begin{theoremx}\label{main-theorem}
    Equality in \eqref{mean-curv-ineq} holds if and only if $(M,g)$ is isometric to a closed hyperbolic geodesic ball $\closedBall_\rho\subset\Hyp^n$ of radius $\rho$, with $\sinh{\rho} = \mathrm{Rad}(\Sigma)$.
\end{theoremx}

\section{Notation}
Given a Riemannian spin manifold $(M, g)$ of dimension $n$, we let $\SpinBdlv M$ denote the spinor bundle associated to an irreducible representation of the complexified Clifford algebra $\mathbb{C}\ell_n$. A section $\varphi \in \Gamma(\SpinBdlv M)$ is said to be an imaginary Killing spinor if it is parallel with respect to the connection $\widehat{\nabla}$ on $\SpinBdlv M$, defined by
\begin{equation}
    \widehat{\nabla}_X \psi : = \nabla_X \psi + \frac{i}{2} X \cdot \psi.
\end{equation}
In this article, we will always let $\nu^M$ denote the outward-pointing unit normal for a manifold $M$ with boundary. The second fundamental form $\II$ of a hypersurface $Q \subset (M, g)$ with unit normal $\nu$ is defined with the sign convention
\[
    \II(X, Y) := g(\nabla_X \nu, Y).
\]
We let $\Shape$ denote the corresponding shape operator, and we put $\mean:= \tr \II$. We write $\FullCurv$ to denote the Riemann curvature $(1,3)$-tensor of $(M, g)$ with the sign convention that $g( R(X,Y)Y,X)\geq 0$ on round spheres. Let $\Curv$ be the corresponding curvature on the spinor bundle $\SpinBdlv M$. These are related via the standard formula
\begin{equation}\label{spin-curvature}
    \Curv_{X,Y} = \frac{1}{4} \sum_{j,k=1}^{n} g( \FullCurv(X,Y)e_j,e_k)\, e_j \cdot e_k,
  \end{equation}
where $\{e_j\}_{j=1}^n$ is an orthonormal basis of $T_pM$ for $p\in M$.

\section{Proof of Theorem \ref{main-theorem}}

Our contribution is the following Proposition.
\begin{proposition}\label{main-prop}
    Let $(M^n,g)$ be a compact, connected Riemannian spin manifold of dimension $n\geq 3$ with connected boundary $\Sigma$. Assume that the complex spinor bundle $\SpinBdlv M$ admits a global frame of imaginary Killing spinors. Assume moreover that $(\Sigma,g|_\Sigma)$ is isometric to the round sphere $\Sphere^{n-1}(r)$ and that, with respect to the outward unit normal, its mean curvature is
\[
    \mean=(n-1)\sqrt{1+r^{-2}}.
\]
Then, with $r=\sinh\rho$, we have that $(M,g)$ is isometric to the closed geodesic ball $\closedBall_\rho\subset\Hyp^n$.
\end{proposition}

\begin{proof}
Let $\widehat{\Curv}$ denote the curvature of $\hatnabla$. Since the sections $s_1,\ldots,s_m$ form a pointwise basis and are $\hatnabla$-parallel, one has $\widehat{\Curv}_{X,Y}=0$ for all vector fields $X,Y$. On the other hand, we compute that 
\[
    \widehat{\Curv}_{X,Y}
    =
    \Curv_{X,Y}
    -\frac{1}{4}[X\mathbin{\cdot},Y\mathbin{\cdot}], 
\]
and so by \eqref{spin-curvature}, we have that
\[
    \sum_{j,k=1}^{n} g( \FullCurv(X,Y)e_j,e_k)\, e_j \cdot e_k = [X\mathbin{\cdot},Y\mathbin{\cdot}].
\]
Set $X = e_a$ and $Y = e_b$, with $a \neq b$, and compute that 
\[
    \sum_{j,k=1}^{n} g(\FullCurv(e_a,e_b)e_j,e_k)\, e_j \cdot e_k = [e_a\cdot,e_b\cdot] = e_a \cdot e_b - e_b \cdot e_a  = 2e_a \cdot e_b.
\]
Since the map $\Lambda^2TM \to \mathrm{End}(\SpinBdlv M)$ given by Clifford multiplication is injective, the antisymmetries of $\FullCurv$ imply that 
\[
    \FullCurv(e_a, e_b) e_j = \delta_{aj}e_b - \delta_{bj}e_a, 
\]
which means that the curvature tensor is given by
\[
    \FullCurv(X,Y)Z
    =
    -\bigl(g(Y,Z)X-g(X,Z)Y\bigr).
\]
Hence
\[
    \sect_M\equiv -1.
\]

At a point $x\in \Sigma$, let $e_1,\ldots,e_{n-1}$ be an orthonormal basis of principal directions at $T_x\Sigma$ with principal curvatures $\kappa_1,\ldots,\kappa_{n-1}$. Since $\Sigma$ is round of radius $r$, the Gauss equation gives, for $k\neq \ell$,
\[
    r^{-2} =
    \sect_{\Sigma}(e_k \wedge e_\ell)=
    \sect_M(e_k\wedge e_\ell)+\kappa_k\kappa_\ell
    =
    -1+\kappa_k\kappa_\ell.
\]
Thus
\begin{equation}
    \kappa_k\kappa_\ell=1+r^{-2}
\end{equation}
for all $k\neq\ell$. Put $c:=\sqrt{1+r^{-2}}=\coth\rho$. If $n\geq 4$, the relations $\kappa_k\kappa_\ell=c^2$ imply that all principal curvatures are equal, and the identity $\kappa_1+\cdots+\kappa_{n-1}=(n-1)c$ gives $\kappa_k=c$ for every $k$. If $n=3$, then
\[
    \kappa_1+\kappa_2=2c,
    \qquad
    \kappa_1\kappa_2=c^2,
\]
so again $\kappa_1=\kappa_2=c$. Therefore $\II=\coth\rho\,g|_\Sigma$.
For the outward unit normal $\nu^M$ of~$M$ and for sufficiently small $\varepsilon > 0$, define
\begin{align*}
    \Phi_M : \Sigma \times [0, \varepsilon) &\to M \\
    (x,t) &\mapsto \exp^M_x(-t\nu^M(x)).
\end{align*}
Choose an isometry $f:(\Sigma,g|_\Sigma)\longrightarrow\bd\Ball_\rho\subset\Hyp^n$ and set
\[
    E:=\Hyp^n\setminus\interior{\Ball_\rho},
  \]
where $\interior{\Ball_\rho}$ denotes the open geodesic ball in $\Hyp^n$ of radius $\rho$.  
For the outward unit normal $\nu^E$  of $E$ and for sufficiently small $\varepsilon > 0$, define
\begin{align*}
    \Phi_E : \Sigma \times [0, \varepsilon) &\to E \\
    (x,t) &\mapsto \exp^E_{f(x)}(-t\nu^E(f(x))).
\end{align*}
We glue $M$ and $E$ along $f$ using the collars given by $\Phi_M$ and $\Phi_E$ to obtain the smooth manifold
\[
    N:=M\cup_f E.
\]
It remains to check that one can also smoothly glue the two metrics along the hypersurface $\Sigma$. Fix $x \in \Sigma$ and use the notation $\gamma(t) = \Phi_M(x, t)$. For $v \in T_x\Sigma$, the corresponding normal Jacobi field $V(t) = \D\Phi_M|_{(x,t)}(v,0)$ satisfies
\[
    \frac{D^2V}{dt^2} + \FullCurv(V, \gamma') \gamma' = 0, \qquad V(0) = v, \qquad V'(0) = -\Shape(v) = -\coth(\rho)v.
\]
Let \(P_t\) denote parallel transport along \(\gamma\). Since \(M\) has
constant sectional curvature \(-1\), the Jacobi equation for a Jacobi field of the form $V(t)=a(t)P_t(v)$ reduces to $a''(t)-a(t)=0$, and the initial conditions become $a(0)=1, a'(0)=-\coth(\rho)$. Consequently, we have that
\[
    a(t)
    =
    \cosh(t)-\coth(\rho)\sinh(t)
    =
    \frac{\sinh(\rho-t)}{\sinh(\rho)},
\]
and so it follows that
\[
    \D\Phi_M|_{(x,t)}(v,0)
    =
    V(t)
    =
    \frac{\sinh(\rho-t)}{\sinh(\rho)}P_t(v).
\]
Hence, we have that
\[
    \Phi_M^*g
    =
    \DD t^2
    +
    \left(\frac{\sinh(\rho-t)}{\sinh\rho}\right)^2g|_\Sigma.
\]

On the exterior side $E$, let $s\geq 0$ be the normal distance from $\bd\Ball_\rho$. Writing the hyperbolic metric in polar coordinates, we see that the hyperbolic metric in the corresponding collar given by $\Phi_E$ is given by
\[
    \D s^2
    +
    \left(\frac{\sinh(\rho+s)}{\sinh\rho}\right)^2g|_\Sigma.
\]
Hence, we see that on both sides of the gluing hypersurface, the metric is given by the same formula $\D u^2+\left(\frac{\sinh(\rho+u)}{\sinh\rho}\right)^2g|_\Sigma$. Hence the glued metric on $N$ is smooth. Also note that $\sect_N\equiv -1$. Moreover, $N$ is complete, because $M$ is compact and~$N$ agrees outside a compact set with the complete hyperbolic exterior $E$. Hence, by the classification of complete hyperbolic space forms, there exists a discrete subgroup $\Gamma \subset \mathrm{Isom}(\Hyp^n)$
such that $N$ is isometric to $\Hyp^n/\Gamma$. Let $q:\Hyp^n\longrightarrow N$ be the universal covering. Since $n\geq 3$, the exterior $\interior{E}=\Hyp^n\setminus\closedBall_\rho$ is simply connected. Hence every connected component $\widehat{E}$ of $q^{-1}(\interior{E})$ is mapped isometrically onto $\interior{E}$. By the standard extension property for local isometries between connected open subsets of hyperbolic space, there is a geodesic ball $\closedBall \subset\Hyp^n$ such that
\[
    \widehat E=\Hyp^n\setminus \closedBall.
\]

If $\gamma\in\Gamma$ is nontrivial, then $\gamma\widehat E$ and $\widehat E$ are distinct components of $q^{-1}(\interior{E})$ and hence they are disjoint. On the other hand, we have that $\widehat E\cap\gamma\widehat E = \Hyp^n\setminus(\closedBall\cup\gamma \closedBall)$, which is a nonempty set because two bounded geodesic balls cannot cover $\Hyp^n$. Therefore, $\Gamma$ is trivial, and it follows that $N\cong\Hyp^n$. Under this isometry, the subset~$\interior{E}$ is sent to the exterior of a closed geodesic ball of radius $\rho$. The complement of $\interior{E}$ in $N$ is $M$, so we have that
\[
    (M,g)\cong\closedBall_\rho\subset\Hyp^n,
\]
which finishes the proof.
\end{proof}

\begin{proof}[Proof of Theorem \ref{main-theorem}]
    By \cite[Corollary 7]{brendle-fill-in}, the spinor bundle $\SpinBdlv M$ admits a frame consisting of imaginary Killing spinors $\{s_{\alpha}\}_{\alpha = 1}^{m}$, $m = 2^{\lfloor \frac{n}{2} \rfloor}$. In the proof of \cite[Corollary 7]{brendle-fill-in}, it is shown that $(\Sigma, g\vert_{\Sigma}) \cong \Sphere^{n-1}(r)$, where $r := \rad(\Sigma)$. By \cite[Proposition 5]{brendle-fill-in}, combined with \eqref{bar-bound}, we get for each $\alpha \in \{1, \hdots, m\}$ that 
    \begin{align}\label{prop5}
        \int_\Sigma \mean\abs{s_{\alpha}}^2\,\dS \leq (n-1)\sqrt{1+r^{-2}}\int_\Sigma \abs{s_{\alpha}}^2\,\dS, 
    \end{align}
    and since we are in the equality case of \eqref{mean-curv-ineq}, we have that 
    \begin{align*}
        0 &\leq \int_{\Sigma} (\mean - \inf_{\Sigma} \mean) \abs{s_{\alpha}}^2\,\dS = \int_{\Sigma} (\mean - (n-1)\sqrt{1+r^{-2}}) \abs{s_{\alpha}}^2\,\dS \overset{\eqref{prop5}}{\leq} 0.
    \end{align*}
    Since each $s_{\alpha}$ is nowhere vanishing, we see that $\mean = (n-1) \sqrt{1 + r^{-2}}$. Hence we may apply Proposition \ref{main-prop} and conclude that $(M,g)$ is isometric to the closed geodesic ball $\closedBall_{\rho}$ of radius $\rho$ with $\sinh{\rho}=r=\rad(\Sigma)$. Conversely, if $(M,g)$ is isometric to a closed geodesic ball, then equality in \eqref{mean-curv-ineq} follows from a straightforward calculation.  
\end{proof}

\newpage
\bibliographystyle{alpha}
\bibliography{literature.bib}
\end{document}